\documentclass[11pt]{amsart} %

\usepackage{float}

\usepackage{array,amsmath, enumerate,  url, psfrag}
\usepackage{amssymb, amsaddr, fullpage}
\usepackage{graphicx,subfigure}
\usepackage{color}

\restylefloat{figure}

\newtheorem{theorem}{Theorem}[section]
\newtheorem{lemma}[theorem]{Lemma}

\newtheorem{problem}[theorem]{Problem}

\title{Planar graphs without short even cycles are near-bipartite}

\author{Runrun Liu$^{1,2}$ \hskip 0.15in Gexin Yu$^{2,3}$}

\address{
$^1$\small Department of Mathematics, Zhejiang Normal University, Jinhua, 321004, China.\\
$^2$\small Department of Mathematics and Statistics, Central China Normal University, Wuhan, 430079, China.\\
$^3$\small Department of Mathematics, William \& Mary, Williamsburg, VA, 23185, USA.
}

\thanks{The research was partially supported by NSFC (11728102) and the NSA grant  H98230-16-1-0316.}

\email{gyu@wm.edu}

\begin{document}

\maketitle

\begin{abstract}
A graph is {\em near-bipartite} if its vertex set can be partitioned into an independent set and a set that induces a forest. It is clear that near-bipartite graphs are $3$-colorable. In this note, we show that planar graphs without cycles of lengths in $\{4, 6, 8\}$  are near-bipartite.
\end{abstract}

\section{Introduction}

A graph is {\em k-degenerate} if each subgraph has a vertex of degree at most $k$.  Every $k$-degenerate graph is $(k+1)$-colorable. Borodin \cite{B76} in 1976 suggested to partition the vertex set of a graph into two sets that induce graphs with better degeneracy properties. A graph $G$ is $(a,b)$-partitionable if its vertices can be partitioned into sets $A$ and $B$ such that the subgraph induced by $A$ is $a$-degenerate and the subgraph induced by $B$ is $b$-degenerate.   Borodin~\cite{B76} conjectured that every planar graph, which is 5-degenerate, is $(1,2)$- and $(0, 3)$-partitionable.   Thomassen~\cite{T95, T01} confirmed these conjectures.

Clearly a graph is bipartite if and only if it is $(0,0)$-partitionable.  A graph is called {\em{ near-bipartite}} if it is $(0,1)$-partitionable. 
Recognizing near-bipartite graphs can be seen as restricted variants of the 3-coloring problem, which is well known to be NP-complete \cite{L73}. Borodin and Glebov \cite{BG01} showed that every planar graph of girth at least $5$ is near-bipartite (see \cite{KT09} for an extension of this result). Dross, Montassier, and Pinlou \cite{DMP15} asked whether every triangle-free planar graph is near-bipartite. 
Borodin, Glebov, Raspaud, and Salavatipour \cite{BGRS05} proved that planar graphs without cycles of length from 4 to 7 are 3-colorable. Wang and Chen \cite{WC07} showed that planar graphs without $\{4,6,8\}$-cycles are 3-colorable. In this paper, we improved this result by showing that planar graphs without $\{4,6,8\}$-cycles are near-bipartite. Some tricks in the proof resemble those first appeared in \cite{BG01, BGRS05} and further developed in \cite{BGMR09, BGR10}.

\begin{theorem}\label{main1}
Every planar graph without \{4,6,8\}-cycles is near-bipartite.
\end{theorem}

An {\em{IF-coloring}} of a graph is a partition of its vertices into two parts such that one part colored $I$ induces an independent set and the other part colored $F$ induces a forest. Given a graph $G$ and a cycle $C$ in $G$, an $IF$-coloring $\phi_C$ of $G[V(C)]$ {\it superextends} to $G$ if there exists an $IF$-coloring $\phi_G$ of $G$ that extends $\phi$ with the property that there is no path joining two vertices of $C$ all of whose vertices are colored $F$ and do not belong to $C$.  We say that $C$ is {\it superextendable} to $G$ if every $IF$-coloring $\phi_C$ of $G[V(C)]$ superextends to $G$. When we wish to specify $G$, we will say $(G,C)$ is superextendable.

Instead of Theorem~\ref{main1}, we actually proved a stronger result:

\begin{theorem}\label{main2}
For every planar graph $G$ without \{4,6,8\}-cycles and any cycle $C$ in $G$ of length at most 12, $(G,C)$ is superextendable.
\end{theorem}

By \cite{BG01}, every planar graph of girth $5$ has an $IF$-coloring. Since $G$ contains no $4$-cycles, $G$ must have a triangle. Thus Theorem~\ref{main1} follows from Theorem~\ref{main2} easily.

Graphs mentioned in this paper are all simple. Let $d(v)$ be the degree of $v$. A $k$-vertex ($k^+$-vertex, $k^-$-vertex) is a vertex of degree $k$ (at least $k$, at most $k$). The same notation will be applied to faces and cycles. Let $C$ be a cycle of a plane graph $G$. We use $int(C)$ and $ext(C)$ to denote the sets of vertices located inside and outside $C$, respectively. The cycle $C$ is called a {\em separating cycle} if $int(C)\ne\emptyset\ne ext(C)$. We still use $C$ to denote the set of vertices of $C$.  An $F$-cycle is a cycle whose vertices are all colored $F$,  and an $F$-path is path whose vertices are colored $F$.

\section{Reducible Configurations}\label{ReducibleConfigurations}
Let $(G, C_0)$ be a counterexample to Theorem~\ref{main2} with minimum $\sigma(G)=|V(G)|+|E(G)|$, where $C_0$ is a cycle of length at most 12 in $G$ that is precolored. If $C_0$ is a separating cycle, then $C_0$ is superextendable in both $G-ext(C_0)$ and $G-int(C_0)$. Thus, $C_0$ is superextendable in $G$, contrary to the choice of $C_0$. Thus, we may assume that $C_0$ is the boundary of the outer face of $G$ in the rest of this paper. 
Call a vertex $v$ {\em internal} if $v\notin C_0$, call a face $f$ {\em internal} if $f\ne C_0$ and {\em truly internal} if $V(f)\cap V(C_0)=\emptyset$.

\begin{lemma}\label{minimum}
Every internal vertex in $G$ has degree at least $3$.
\end{lemma}

\begin{proof}
Suppose there is an internal $2^-$-vertex $v$ in G. By minimality of $(G,C_0)$, the $IF$-coloring of $C$ can superextend to $G-v$. If $v$ has a neighbor colored $I$, then color $v$ with $F$, otherwise, color $v$ with $I$. In all cases, $C_0$ superextends to $G$, a contradiction.
\end{proof}

\begin{lemma}\label{separating}
The graph $G$ has no separating cycle of length at most 12.
\end{lemma}

\begin{proof}
Suppose otherwise that $C$ is a separating cycle of length at most 12 in $G$. Then $C$ is inside of $C_0$. By the minimality of $(G,C_0)$, $(G-int(C),C_0)$ is superextendable, and after that, $G[V(C)]$ is colored. By the minimality of $(G,C_0)$ again, $(C \cup int(C),C)$ is superextendable. Thus, $(G,C_0)$ is superextendable, a contradiction.
\end{proof}

\begin{lemma}\label{C}
The following are true about $C_0$.
\begin{enumerate}
\item $C_0$ is chordless.
\item Every non-adjacent pair $v_0$ and $v_1$ on $C_0$ have no internal common neighbors.
\end{enumerate}
\end{lemma}

\begin{proof}
(1) If $C_0$ has a chord, then by Lemma~\ref{separating}, $V(C_0)=V(G)$. But by definition,  the precoloring of $C_0$ is already a good IF-coloring of $G$.

(2) Let $w$ be an internal common neighbor of $v_0,v_1\in V(C_0)$ with $v_0v_1\not\in E(G)$. Let $P_1$ and $P_2$ be the two paths of $C_0$ with ends $v_0$ and $v_1$. Then both $P_1+v_0wv_1$ and $P_2+v_0wv_1$ are cycles of length at most 12. By Lemma~\ref{separating}, $N(w)\subset V(C_0)$. Note that $d(w)\ge3$ by Lemma~\ref{minimum}. So $w$ has at least one neighbor, by symmetry say $v_2$ is on the segment of $C_0$ from $v_1$ to $v_0$ in the clockwise order. Let $P_i'$ be the path of $C_0$ with ends $v_i,v_{i+1}$ for $i\in[3]$ (index module $3$) in the clockwise order. Let $C_i=P_i'+v_iwv_{i+1}$ for $i\in[3]$. Note that $|C_1|+|C_2|+|C_3|\le18$. Without loss of generality, let $|C_1|\le|C_2|\le|C_3|$. Since $G$ contains no $\{4,6,8\}$-cycles, $|C_1|\in\{3,5\}$;  and when $|C_1|=3$, $|C_2|\ge 9$; when $|C_1|=5$, $|C_2|\ge7$. In any case, $\min\{3+9+9, 5+7+7\}>18$, a contradiction.
\end{proof}

\begin{lemma}\label{5face}
Every truly internal $5$-face is incident to at most four $3$-vertices.
\end{lemma}

\begin{proof}
Suppose otherwise that $f=[v_1v_2v_3v_4v_5]$ is a truly internal $5$-face with five $3$-vertices. Let $u_i$ be the neighbor of $v_i$ not on $f$ for $i\in[5]$.  Let $G'$ be the graph obtained by identifying $v_3$ and $u_1$ of $G-\{v_1,v_2,v_4,v_5\}$.  Note that if the identification creates an edge between vertices of $C_0$,  then $u_1, u_3\in V(C_0)$, but the path between them along $C_0$ with $u_1v_1v_2v_3u_3$ forms a separating cycle of length at most $12$,  a contradiction to Lemma~\ref{separating}. It follows that the precoloring of $C_0$ remains valid in $G'$. Furthermore, if there is a path $Q$ of length at most 8 between $v_3$ and $u_1$, then $G$  has a separating cycle of length at most $11$ which is obtained from the path $u_1v_1v_2v_3$  and $Q$, contrary to Lemma~\ref{separating}. So $G'$ contains no $\{4,6,8\}$-cycles. Since $|V(G')|<|V(G)|$, we can extend the coloring of $C_0$ to $G'$.   Now we can extend this coloring to an $IF$-coloring of $G$ in the following way.

First assume that the new vertex created by the identification is colored $I$. Then we color $u_1$ and $v_3$ with $I$ and $v_1,v_2,v_4$ with $F$. If $u_5$ is colored with $F$, then color $v_5$ with $I$. If $u_5$ is colored with $I$, then color $v_5$ with $F$, which is invalid only if $u_2$ and $u_4$ are both colored $F$, in which case, we recolor $v_3$ with $F$ and $v_2,v_4$ with $I$. Note that we  create neither $F$-cycle nor $F$-path between two vertices of $C_0$.

Now assume that the new vertex created by the identification is colored $F$. First we color $u_1$ and $v_3$ with $F$. Note that there exists no $F$-path between $u_1$ and $v_3$ in $G$. If $u_4$ is colored $I$, then color $v_1,v_4$ with $F$, and for each $i\in\{2,5\}$, color $v_i$ with a color different from $u_i$; this is not a valid coloring only if $u_2$ and $u_5$ are both colored $I$, in which case we recolor $v_1$ with $I$. So we may assume that $u_4$ is colored $F$.  Color $v_1, v_4$ with $I$ and $v_2, v_3, v_5$ with $F$, we obtain a valid coloring, unless we create an $F$-cycle or an $F$-path between two vertices of $C_0$. In the bad cases, $u_2$ and $u_3$ are colored $F$. We may further assume that $u_5$ is colored $F$, for otherwise we can color $v_2, v_4$ with $I$ and $v_1, v_3, v_5$ with $F$ to obtain a valid coloring. 

Consider the  case that $u_i$ is colored $F$ for $i\in [5]$.   Note that there is no $F$-path between $u_1$ and $u_3$, and there cannot exist $F$-path between $u_1$ and $C_0$ and between $u_3$ and $C_0$ at the same time, since $G'$ has a valid coloring.    By symmetry, we may assume that there is no $F$-path between $u_3$ and $C_0$.  There must be an $F$-path, say $P$, between $u_2$ and $u_3$, for otherwise we obtain a valid coloring by coloring $v_2, v_3, v_5$ with $F$ and $v_1, v_4$ with $I$.  We color $v_1, v_2, v_4$ with $F$ and $v_3, v_5$ with $I$.  Since the $F$-path $u_1v_1v_2u_2$ in $G$ can be replaced with the $F$-path $u_1(v_3)u_3Pu_2$ in $G'$, this cannot be $F$-cycle or $F$-path containing $u_1v_1v_2u_2$ in $G$. Therefore we obtain a valid coloring.
\end{proof}

A $3$-vertex $v\not\in C_0$ is {\em bad} if $v$ is on a $3$-face. A {\em tetrad} in a plane graph is a path $v_1v_2v_3v_4$ of four internal $3$-vertices contained in the boundary of a face, so that both $v_1v_2$ and $v_3v_4$ are edges of triangles.

\begin{lemma}\label{tetrad}
G contains no tetrad. Consequently, no face of $G$ is incident with five consecutive bad vertices. Furthermore, if a face of $G$ is incident with consecutive vertices $v_0,v_1,\ldots,v_5$ and the vertices $v_1,\ldots,v_4$ are bad, then the edges $v_0v_1,v_2v_3$ and $v_4v_5$ are incident with triangles.
\end{lemma}

\begin{proof}
Let $v_1v_2v_3v_4$ be a tetrad in $G$. Let $N(v_1)=\{x,v_2,v_1'\}$ and $N(v_4)=\{y,v_3,v_4'\}$, where $x,y$ are not in the triangles. Let $G'$ be the graph obtained by identifying $y$ and $v_1'$ of $G-\{v_1,v_2,v_3,v_4\}$. Note that the identification does not create a chord in $C_0$ or identify two vertices of $C_0$. For otherwise, there exists a path $P$ of length $4$ or $5$, which is internally disjoint from $C_0$, between two vertices of $C_0$, and since $|C_0|\le12$, there exists a cycle in $P\cup C_0$ of length at most $11$ that separates $x$ and $v_4$, contrary to Lemma~\ref{separating}. There is no path $Q$ of length at most 8 between $y$ and $v_1'$, for otherwise $G$ would have a cycle separating $K$ of length at most 12 that is obtained from the path $yv_3v_2v_1v_1'$  and $Q$,  contrary to Lemma~\ref{separating}. Thus, no new $8^-$-cycles are created in $G'$. This implies that $G'$ contains no $\{4,6,8\}$-cycles. Since $|V(G')|<|V(G)|$,  the precoloring of $C_0$ can be extended to  an $IF$-coloring of $G'$. Now we extend it to an $IF$-coloring of $G$.

First we color $v_1'$ and $y$ with the color of the identified vertex.  If $v_1',y$ are colored $I$, then color $v_1,v_2,v_4$ with $F$ and color $v_3$ with a color different from the color of $v_4'$, and we obtain a valid coloring.  So we assume that $v_1',y$ are colored $F$. Then there exists no $F$-path between $v_1'$ and $y$ in $G$.  Assume first that $v_4'$ is colored $I$. If $x$ is colored $I$, then color $v_1,  v_3, v_4$ with $F$ and $v_2$ with $I$; if $x$ is colored $F$, then color $v_1$ with $I$ and $v_2, v_3, v_4$ with $F$. In either case, we obtain a valid coloring.  Let $v_4'$ be colored $F$.  If $x$ is colored $I$, then color $v_2, v_4$ with $I$ and $v_1, v_3$ with $F$, and we obtain a valid coloring. So we assume $x$ is also colored $F$.   Since there is no $F$-path between $v_1'$ and $y$,  either there is no $F$-path between $v_1'$ and $v_4'$ or there is no $F$-path between $v_4'$ and $y$. Also, there is at most one $F$-path between $\{v_1', y\}$ and $C_0$. If there is no $F$-path between $v_1'$ and $C_0$ nor between $v_1'$ and $v_4'$, then color $v_1, v_4$ with $I$ and $v_2,v_3$ with $F$; if there is no $F$-path between $y$ and $C_0$ nor between $y$ and $v_4'$, then color $v_1, v_3$ with $I$ and $v_2,v_4$ with $F$; in either of these cases, we obtain a valid coloring. We consider the other two cases.
\begin{itemize}
\item Case 1: no $F$-path between $v_1'$ and $C_0$ nor between $y$ and $v_4'$.  We may assume that there are $F$-paths between $v_1'$ and $v_4'$ (say $P$) and between $y$ and $C_0$.  Since the $F$-path $P$ in $G$ can be replaced with the $F$-path $y(v_1')Pv_4'$ in $G'$, there is no $F$-path between $v_4'$ and $C_0$. Then we can color $v_1, v_3$ with $I$ and $v_2, v_4$ with $F$ to obtain a valid coloring.
\item Case 2: no $F$-path between $y$ and $C_0$ nor between $v_1'$ and $v_4'$.   We may assume that there are $F$-paths between $v_1'$ and $C_0$ and between $v_4'$ and $y$ (say $P$).   Since the $F$-path $P$ in $G$ can be replaced with the $F$-path $v_1'(y)Pv_4'$ in $G'$, there is no $F$-path between $v_4'$ and $C_0$. Then we can color $v_1, v_4$ with $I$ and $v_2, v_3$ with $F$ to obtain a valid coloring.
\end{itemize}


Now suppose that $f$ is incident with five consecutive bad vertices $v_1,\ldots,v_5$. Since $v_3$ is on a $3$-face and $G$ contains no adjacent $3$-faces, either $v_2v_3$ or $v_3v_4$ is an edge on a $3$-face. In the former case, $v_4v_5$ is an edge on a 3-face; in the latter case,  $v_1v_2$ must be an edge of a $3$-face. This implies that either $v_2 v_3v_4v_5$ or $v_1v_2v_3v_4$ is a tetrad, a contradiction.

For the furthermore part,  if a face of $G$ is incident with consecutive vertices $v_0,v_1,\ldots,v_5$ and the vertices $v_1,\ldots,v_4$ are bad, then either each of $v_1v_2,v_3v_4$ is an edge of a 3-face or each of $v_0v_1,v_2v_3,v_4v_5$ is an edge incident with a triangle.  But  the former cannot happen, as the path $v_1v_2v_3v_4$ is a tetrad and each of $v_1,\ldots, v_4$ is a bad vertex.
\end{proof}

\section{Discharging Procedure}
\label{Discharging}

We are now ready to present a discharging procedure that will complete the proof of Theorem~\ref{main2}.  Let $x\in V(G)\cup (F(G)-C_0)$ have an initial charge of $\mu(x)=d(x)-4$, and $\mu(C_0)=|C_0|+4$. By Euler's Formula, $\sum_{x\in V\cup F}\mu(x)=0$.
Let $\mu^*(x)$ be the charge of $x\in V\cup F$ after the discharge procedure. To lead to a contradiction, we shall prove that $\mu^*(x)\ge 0$ for all $x\in V(G)\cup F(G)$ and $\mu^*(C_0)>0$.

Let $v$ be a $4$-vertex on a face $f$. The vertex $v\not\in C_0$ is {\em poor} to $f$ if  either $v$ is incident with a $3$-face that is not adjacent to $f$, or $v$ is incident with two $3$-faces both adjacent to $f$, or $v$ is incident with a $5$-face adjacent to $f$. A $2$-vertex $u\in C_0$ is {\em special} if $u$ is on an internal $5$-face.

\medskip

\noindent Here are the discharging rules:

\begin{enumerate}[(R1)]
\item Each $3$-face gets $\frac{1}{3}$ from each incident vertex.

\item Each internal $5$-face gets $\frac{1}{3}$ from each incident $4^+$-vertex and gives $\frac{1}{3}$ to each incident $2$-vertex or internal $3$-vertex. Each internal $7^+$-face gives $\frac{2}{3}$ to each incident $2$-vertex or bad $3$-vertex, $\frac{1}{3}$ to each incident internal non-bad $3$-vertex or poor $4$-vertex, and then each internal face gives the surplus charge to $C_0$.

\item Each $4^+$-vertex on $C_0$ gives $\frac{1}{3}$ to each incident internal $5$-face.

\item The outer face $C_0$ gives $\frac{5}{3}$ to each incident special $2$-vertex, $\frac{4}{3}$ to each incident non-special $2$-vertex or $3$-vertex on a $3$-face,  and $1$ to each other incident vertex.
\end{enumerate}

\begin{lemma}\label{vertex}
Every vertex $v$ in $G$ has nonnegative final charge.
\end{lemma}

\begin{proof}
We consider the degree of $v$.

Let $d(v)=2$.  Then by Lemma~\ref{minimum} $v\in C_0$.  By (R2) and (R4),  $v$ gets $\frac{5}{3}$ from $C_0$ and $\frac{1}{3}$ from the other incident face when $v$ is special, or gets $\frac{4}{3}$ from $C_0$ and $\frac{2}{3}$ from the other incident face otherwise. Thus, $\mu^*(v)\ge2-4+\min\{\frac{5}{3}+\frac{1}{3}, \frac{4}{3}+\frac{2}{3}\}=0$.

Let $d(v)=3$. If $v\in C_0$, then by (R1) and (R4),  $v$ gets $\frac{4}{3}$ from $C_0$ if $v$ is on a $3$-face,  or gets $1$ from $C_0$ otherwise. So $\mu^*(v)\ge3-4+\min\{1,\frac{4}{3}-\frac{1}{3}\}=0$. Let $v\not\in C_0$. By (R1) and (R2), $v$ gets $\frac{2}{3}$ from each of the two incident non-triangular faces and gives $\frac{1}{3}$ to its incident $3$-face if $v$ is on a $3$-face, or $v$ gets $\frac{1}{3}$ from each of the three incident faces if $v$ is not bad. So $\mu^*(v)\ge3-4+\min\{\frac{1}{3}\cdot3,\frac{2}{3}\cdot2-\frac{1}{3}\}=0$.

Let $d(v)=4$. Let $v\not\in C_0$.  By (R1) and (R2), $v$ gives $\frac{1}{3}$ to each of (at most two) incident $5^-$-faces and gets $\frac{1}{3}$ from each of the two incident $7^+$-faces when $v$ is on two $3$-faces; gets $\frac{1}{3}$ from the incident face when $v$ is on exactly one $3$-face and not on a $5$-face; gets $\frac{1}{3}$ from each of the two incident $7^+$-faces adjacent to the $5$-face when $v$ is on a $5$-face; so $\mu^*(v)\ge 4-4+\min\{\frac{1}{3}\cdot2-\frac{1}{3}\cdot2, \frac{1}{3}-\frac{1}{3}\}=0$. When $v\in C_0$,  $v$ gets $1$ from $C_0$ and gives $\frac{1}{3}$ to each of at most three incident faces by (R3), so $\mu^*(v)\ge4-4+1-\frac{1}{3}\cdot3=0$.

Finally, let $d(v)\ge5$.  If $v\in C_0$, then $v$ gets $1$ from $C_0$ and gives $\frac{1}{3}$ to each of at most $d(v)-1$ incident faces by (R1), (R3) and (R4); if $v\not\in C_0$, then $v$ gives $\frac{1}{3}$ to each of at most $\lfloor\frac{d(v)}{2}\rfloor$ incident $5^-$-faces by (R1) and (R2).  So $\mu^*(v)\ge d(v)-4+\min\{1-(d(v)-1)\cdot\frac{1}{3}, -\lfloor\frac{d(v)}{2}\rfloor\cdot\frac{1}{3}\}>0$.
\end{proof}

\begin{lemma}\label{face}
Every internal face in $G$ has nonnegative final charge.
\end{lemma}

\begin{proof}
Let $f$ be an internal face in $G$.  Recall that $G$ contains no $\{4,6,8\}$-cycles. For $d(f)=3$, $f$ gets $\frac{1}{3}$ from each incident vertex by (R1), so $\mu^*(f)\ge3-4+\frac{1}{3}\cdot3=0$. For $d(f)\ge5$, by (R2) we only need to show that $f$ has nonnegative charge before it sends charge to $C$, without loss of generality, we still use $\mu^*(f)$ to denote it.

Suppose $d(f)=5$. If $V(f)\cap C_0=\emptyset$, then $f$ is incident with at most four $3$-vertices by Lemma~\ref{5face}. By (R2), $f$ gains at least $\frac{1}{3}$ from its incident $4$-vertices. If $V(f)\cap C_0\ne\emptyset$, then $1\le|V(f)\cap C_0|\le3$ by Lemma~\ref{C}. If $|V(f)\cap C_0|=1$, then $f$ and $C_0$ share a $4^+$-vertex, thus by (R3)  $f$ gets $\frac{1}{3}$ from the incident vertex on $C_0$ and by (R2) gives $\frac{1}{3}$ to each of at most four incident internal $3$-vertices. If $|V(f)\cap C_0|=2$, then $f$ contains no $2$-vertices and gives $\frac{1}{3}$ to each of at most three incident internal $3$-vertices. If $|V(f)\cap C_0|=3$, then $f$ shares exactly one $2$-vertex with $C_0$ by Lemma~\ref{C}, thus by (R2), $f$ gives $\frac{1}{3}$ to the $2$-vertex and each of at most two incident internal $3$-vertices. In any case, $\mu^*(f)\ge5-4+\min\{-\frac{1}{3}\cdot4+\frac{1}{3}, -\frac{1}{3}\cdot3\}=0$.

Suppose $d(f)=7$. Then $f$ is not adjacent to any $3$-faces and contains at most four $2$-vertices. If $f$ contains a $2$-vertex, then $f$ shares at least two $3^+$-vertices with $C_0$, and f gives no charge to them. By (R2), $\mu^*(f)\ge7-4+\min\{-\frac{2}{3}\cdot4-\frac{1}{3}, -\frac{1}{3}\cdot7\}=0$.

Suppose $d(f)\ge9$. If $f$ contains $2$-vertices, then $f$ does not give charge to at least two $3^+$-vertices shared with $C_0$, so by (R2), $\mu^*(f)\ge d(f)-4-\frac{2}{3}(d(v)-2)>0$. So assume that $f$ contains no $2$-vertices.  By Lemma~\ref{tetrad}, $f$ contains at least two vertices that are not bad $3$-vertices.  It follows that $\mu^*(f)\ge d(f)-4-\frac{2}{3}\cdot(d(f)-2)-\frac{1}{3}\cdot2\ge -\frac{1}{3}$, with $\mu^*(f)<0$ only if $f$ is a $9$-face with seven bad $3$-vertices and two other internal non-bad $3$-vertices or poor $4$-vertices. But when $f$ contains seven bad $3$-vertices, at least four bad $3$-vertices are consecutive on $f$.   By Lemma~\ref{tetrad},  there are exactly four bad consecutive vertices on $f=v_1\ldots v_9$, say $v_1, v_2, v_3, v_4$, and $v_9$ and $v_5$ should be $4^+$-vertices.  As $v_6, v_7, v_8$ are bad $3$-vertices, $v_5$ or $v_9$, by symmetry say $v_5$, is not on another $5^-$-face. Then $f$ gives no charge to $v_5$ by (R2). Therefore, there is no $9$-face with negative final charge.
\end{proof}

{\bf Proof of Theorem~\ref{main2}}. By Lemma~\ref{vertex} and Lemma~\ref{face}, it is sufficient for us to check that $C_0$ has positive final charge. By Lemma~\ref{C}, if a $5$-face shares $2$-vertices with $C_0$, then it shares exactly one 2-vertex with $C_0$.  Let $P$ be a maximal path of $C_0$ such that each vertex in $V(P)$ is either a special $2$-vertex or a neighbor of a special $2$-vertex.  By (R4), $C_0$ gives $\frac{5}{3}$ to each incident special $2$-vertex and $1$ to each neighbor of a special $2$-vertex on $C$. So $C_0$ gives $\lceil\frac{|V(P)|}{2}\rceil\cdot1+\lfloor\frac{|V(P)|}{2}\rfloor\cdot\frac{5}{3}\le\frac{4}{3}|V(P)|$ to vertices on $P$. By (R4), $C_0$ gives at most $\frac{4}{3}$ to each vertex not in such a path.  It follows that $C_0$ gives at most $\frac{4}{3}|C_0|$ to its vertices.  So $\mu^*(C_0)\ge |C_0|+4-\frac{4}{3}|C_0|=\frac{1}{3}(12-|C_0|)\ge 0$, and $\mu^*(C_0)=0$ only when $|C_0|=12$ and either every other vertex on $C_0$ is a special $2$-vertex or each vertex on $C_0$ is either a $2$-vertex or $3$-vertex on a $3$-face. In the former case, let $u$ be a special $2$-vertex on $C_0$ and $f'$ be the internal $5$-face containing $u$. Let $u_1$ and $u_2$ be the two neighbors of $u$ on $f'$. Since $G$ contains no $8$-cycles, $u_1$ and $u_2$ are $4^+$-vertices. By (R2) and (R4), $f$ gets $\frac{1}{3}$ from each of $u_1$ and $u_2$ and gives at most $\frac{1}{3}$ to each other incident vertex. So $C_0$ can get at least $5-4-3\cdot\frac{1}{3}+\frac{1}{3}\cdot2=\frac{2}{3}$ from $f'$. So $\mu^*(C_0)>0$. In the latter case, since $G\ne C_0$, $C_0$ contains two bad $3$-vertices that belong to different $3$-faces but on the same face, say $f_0$, other than $C_0$.  It follows that $f_0$ is adjacent to at least two $3$-faces, so it must be a $9^+$-face since $G$ contains no $\{4,6,8\}$-cycles. By (R2), $f_0$ gives out at most $\frac{2}{3}(d(f_0)-2)$ to all incident vertices and gives at least $d(f_0)-4-\frac{2}{3}(d(f_0)-2)\ge\frac{1}{3}$ to $C_0$. So $\mu^*(C_0)>0$.

\section{Final remarks}

Kang, Jin, and Wang \cite{KJW16} proved that planar graphs without \{4,6,9\}-cycles are 3-colorable. We believe that we could get the following result, by slightly modifying the proof of Theorem~\ref{main2}. We will keep it as a future work.

\begin{problem}\label{main3}
Every planar graph without \{4,6,9\}-cycles is near-bipartite.
\end{problem}

{\bf A sketch of proof of Problem~\ref{main3}}.  We use the same discharging rules as the proof of Theorem~\ref{main2}.   The main difference is that we now have $8$-faces to consider but have no $9$-faces. In the reducible configurations part, we need to consider the property of the outer face $C_0$ carefully since the forbidden cycles have changed. In the discharging part, Now a $7$-face $f$ may share an edge with one $3$-face, but not more. If $f$ contains a $2$-vertex, then it contains at least two $3^+$-vertices on $C_0$, thus gives no charge to them.  As $f$ contains at most three $2$-vertices by Lemma~\ref{C}, $\mu^*(f)\ge 7-4-3\cdot \frac{2}{3}-2\cdot \frac{2}{3}=-\frac{1}{3}$. Then $\mu^*(f)<0$ only if $f$ contains three $2$-vertices and one internal triangle with two $3$-vertices on $f$, but we have a separating cycle of length at most $12$ in this case.   For $8$-faces, they do not share edges with triangle and have at most four $2$-vertices, so their final charges are at least $8-4-4\cdot \frac{2}{3}-4\cdot\frac{1}{3}=0$.  The calculation for other vertices and faces is the same as in the proof of Theorem~\ref{main2}. $\Box$

\medskip

We believe that planar graphs without cycles of lengths from $4$ to $7$ are near-bipartite, but failed to prove it.  The trouble is that an $8$-face may be adjacent to too many triangles when we try a similar approach to the proof of Theorem~\ref{main2}, which we cannot reduce. The condition in Theorem~\ref{main2} excludes $8$-cycles, and the conditions in Problem~\ref{main3} excludes the difficult situation.

\bigskip

{\bf Acknowledgement:} the authors would like to thank Daniel Cranston for bringing their attention to this research problem, and the referees for their careful reading and valuable suggestions.

\end{document}